\documentclass[11pt,reqno]{amsart}

\usepackage[utf8]{inputenc}
\usepackage{amsfonts}
\usepackage{amsmath}
\usepackage{mathrsfs}
\usepackage{amsthm}
\usepackage{amssymb}
\usepackage[top=25truemm,bottom=25truemm,left=25truemm,right=25truemm]{geometry}
\usepackage[bookmarks=false,draft=false,breaklinks,colorlinks]{hyperref}
\usepackage[dvipdfmx]{graphicx}
\usepackage[all]{xy}

\usepackage{color}
\usepackage{url}

\usepackage{comment}

\usepackage{fancyhdr}
\usepackage{enumitem}
\setenumerate{label=(\arabic*),nosep}
\setitemize{nosep}
\newlist{clist}{enumerate}{1}
\setlist*[clist]{label=(\roman*),nosep}

\makeatletter
\let\@fnsymbol\@arabic
\makeatother

\theoremstyle{definition}
\newtheorem{Def}{Definition}[section]
\newtheorem{Rem}[Def]{Remark}
\newtheorem{Eg}[Def]{Example}
\newtheorem*{Cla}{Claim}
\newtheorem{Que}[Def]{Question}

\theoremstyle{plain}
\newtheorem{Thm}[Def]{Theorem}

\newtheorem{Prop}[Def]{Proposition}
\newtheorem{Lem}[Def]{Lemma}
\newtheorem{Cor}[Def]{Corollary}

\newtheorem{Nota}[Def]{Notation}

\newcommand{\quotient}[2]{
\mathchoice{  \text{\raise1ex\hbox{$#1$}\!\Big/\!\lower1ex\hbox{$#2$}} }
                  {  \text{\raise1pt\hbox{$#1$}\big/\lower1pt\hbox{$#2$}} }
                  {  {#1}\,/\,{#2}  }
                  {  {#1}\,/\,{#2}  }
}

\title{NC functions over the nc Grassmannian}
\author{Hyuga Ito}
 
\address{
Graduate School of Mathematics, Nagoya University, Furocho, Chikusaku, Nagoya, 464-8602, Japan
}
\email{hyuga.ito.e6@math.nagoya-u.ac.jp}

\date{\today}

\begin{document}

\begin{abstract}
    We explain how to formulate Voiculescu’s non-commutative Riemann sphere framework for fully matricial functions \cite{v10} within the theory of nc functions developed by Vinnikov and Kaliuzhnyi-Verbovetskyi \cite{kvv14}. We then extend this framework from the Riemann sphere to Grassmannians (and flag manifolds). Moreover, as an example of nc functions in this setting, we introduce a generalization of Voiculescu’s non-commutative resolvent on the Riemann sphere, study a corresponding generalization of the resolvent equation, and discuss aspects of the spectral analysis of unbounded operators in Voiculescu’s framework \cite{v10}.
\end{abstract}
\maketitle


\allowdisplaybreaks{
\section{Introduction}
Since Taylor's work \cite{t72}, the question of what constitutes an appropriate notion of function of several non-commuting variables has been considered. Taylor \cite[section 6]{t72} introduced its candidate, which is now called an \textit{nc function} (nc stands for non-commutative), and Vinnikov and Kaliuzhnyi-Verbovetskyi \cite{kvv14} later developed the theory of nc functions systematically. Let us briefly recall some notions of the theory of nc functions.

Let $\mathcal{R}$ be a unital commutative ring, and let $\mathcal{M},\mathcal{N}$ be $\mathcal{R}$-modules. Set $M(\mathcal{M})=\bigsqcup_{n\in\mathbb{Z}_{\geq1}}M_{n}(\mathcal{M})$ with $M_{n}(\mathcal{M})=\mathcal{M}\otimes_{\mathcal{R}}M_{n}(\mathcal{R})$, $n\in\mathbb{Z}_{\geq1}$. An object $\Omega$ is called an \textit{nc set over $M(\mathcal{M})$} if $\Omega$ is a subset of $M(\mathcal{M})$ and $a\oplus b\in\Omega$ for any $a,b\in\Omega$. A map $f$ from an nc set $\Omega$ over $M(\mathcal{M})$ to $M(\mathcal{N})$ is \textit{graded} if $f(\Omega_{n})\subset M_{n}(\mathcal{N})$ for any $n\in\mathbb{Z}_{\geq1}$, where $\Omega_{n}$ denotes $\Omega\cap M_{n}(\mathcal{M})$. A graded map $f:\Omega\to M(\mathcal{N})$ is \textit{nc} if $f(a\oplus b)=f(a)\oplus f(b)$ for any $a,b\in\Omega$ and $f(sas^{-1})=sf(a)s^{-1}$ for any $a\in\Omega\cap M_{n}(\mathcal{M})$ and $s\in GL_{n}(\mathcal{R})$ with $sas^{-1}\in\Omega$. According to \cite[Proposition 2.1]{kvv14}, a graded function $f:\Omega\to M(\mathcal{N})$ is nc if and only if $f$ satisfies the following \textit{intertwining property}: For any $a\in\Omega_{n}$, $b\in \Omega_{m}$ and $T\in M_{n,m}(\mathcal{R})$,
\begin{equation}
    aT=Tb\quad \Rightarrow\quad f(a)T=Tf(b).\label{eq_affine_intertwining}
\end{equation}
We also have the notion of \textit{nc difference-differential operator} $\Delta$, which plays the role of a kind of ``non-commutative differential operator'' for nc functions.

In the context of free probability, Voiculescu \cite{v00,v04} independently introduced and developed a certain theory of functions of non-commuting variables. 
Its formulation is conceptually similar to that of \cite{kvv14}; namely, it involves corresponding notions of nc sets, nc functions, and the nc difference-differential operator, which are called \textit{fully matricial sets}, \textit{fully matricial functions}, and \textit{Voiculescu's derivation-comultiplication $\partial$}, respectively. However, there are several differences. 

A main difference among the above-mentioned theories lies in how each ``non-commutative differential operator'' is constructed. The nc difference-differential operator $\Delta$ is constructed for all nc functions by repeatedly using the intertwining property (\ref{eq_affine_intertwining}) with appropriate matrices $T$. On the other hand, Voiculescu's derivation-comultiplication $\partial$ is constructed only for analytic functions, relying on their properties (e.g., the linearity of the G\^{a}teaux derivative of an analytic function).  
See \cite[Appendix B]{i23} for a more precise explanation of the relation between Vinnikov and Kaliuzhnyi-Verbovetskyi's theory \cite{kvv14} and Voiculescu's theory \cite{v04}.

The above-mentioned theories are developed over matrix spaces; that is, they are ``\textit{affine} frameworks''.
Voiculescu \cite{v10} further generalized his affine theory \cite{v04} of fully matricial functions to one over the \textit{Riemann sphere}, ``\textit{non-affine}'' space. 
It is then natural to ask whether one can formulate a corresponding non-affine framework of nc functions in the context of \cite{kvv14}. 

The purpose of this paper is to explain how to formulate such a non-affine framework of \cite{kvv14}, particularly in the case of the Grassmannian (and, more generally, flag manifolds). A key observation is that \textit{there exists an appropriate substitute for the intertwining property (\ref{eq_affine_intertwining}) even in a non-affine setting}. Indeed, we will show that, in the case of the Grassmannians, there is such a property (called the \textit{Grassmannian intertwining property}), which recovers the above-mentioned intertwining property (\ref{eq_affine_intertwining}). Just as the intertwining property (\ref{eq_affine_intertwining}) played a central role in \cite{kvv14}, the Grassmannian intertwining property will play a central role in the construction of nc difference-differential operator. 

This paper consists of six sections besides this introduction, and an appendix.
In section \ref{section_grassmann}, we will introduce the concept of nc $(d;m)$-Grassmannian $Gr^{(d;m)}(\mathcal{A})$ over $\mathcal{R}$-algebra $\mathcal{A}$ from the viewpoint of homogeneous spaces. 
In section \ref{section_ncsset_function}, we will define the notions of nc sets and nc functions over the nc $(d;m)$-Grassmannian. Here, we will establish the Grassmannian intertwining property (see Proposition \ref{prop_grassmann_intertwining}).
In section \ref{section_ddoperator}, we will construct an nc difference-differential operator, denoted by $\widetilde{\Delta}^{(d;m)}_{u,v}$, for each $(u,v)\in[m-d]\times[d]$, based on the Grassmannian intertwining property. 
In section \ref{section_Grassmann_resolvent}, we will introduce the $(d;m)$-Grassmannian $\mathcal{B}$-resolvent set $\widetilde{\rho}^{(d;m)}(\pi;\mathcal{B})$ and $(d;m)$-Grassmannian $\mathcal{B}$-resolvent $\widetilde{\mathfrak{R}}^{(d;m)}(\pi;\mathcal{B}|v,u)$ with respect to $(v,u)$, which are examples of nc set and nc function, respectively, and they generalize those of Voiculescu. 
In section \ref{section_moreabout}, we will generalize Voiculescu's resolvent equation \cite[subsection 6.1]{v10} to the $(d;m)$-Grassmannian $\mathcal{B}$-resolvents (see Theorem \ref{thm_nescond}). Moreover, we will explain that Voiculescu's Riemann sphere framework, developed in \cite{v10}, can play a role of the spectral analysis of unbounded operators, and show a partial converse of Theorem \ref{thm_nescond} (see Theorem \ref{thm_partial_sufcond}).
In section \ref{section_question}, we will pose some hopefully natural questions.
In appendix \ref{appendix_flag}, we will explain how to extend the Grassmannian framework to the setting of flag manifolds.

\section{The nc Grassmannian $Gr^{(d;m)}(\mathcal{A})$ over $\mathcal{R}$-algebra $\mathcal{A}$}\label{section_grassmann}
Let $\mathcal{R}$ be a unital commutative ring and $\mathcal{A}$ a unital $\mathcal{R}$-algebra. The scalar multiplication by $\mathcal{R}$ is understood to act centrally; namely, $a\cdot r=r\cdot a$ for all $r\in\mathcal{R}$ and $\mathcal{A}$. Fix $m\in\mathbb{Z}_{\geq1}$, $d\in[m]:=\{1,2,\dots,m\}$. We define 
\begin{align*}
    H^{(d,m)}_{n}(\mathcal{A})
    &=\left\{\left[
    \begin{smallmatrix}
        X&0\\
        Y&Z
    \end{smallmatrix}\right]\,\middle|\,
    X\in GL_{m-d}(M_{n}(\mathcal{A})),Y\in M_{d,m-d}(M_{n}(\mathcal{A})),Z\in GL_{d}(M_{n}(\mathcal{A}))
    \right\}.
\end{align*}

The equivalence relation $\sim^{(d;m)}_{n}$ on $GL_{m}(M_{n}(\mathcal{A}))$ is defined by, for any $A,B\in GL_{m}(M_{n}(\mathcal{A}))$, $A\sim^{(d;m)}_{n}B$ if and only if there exists an element $\Gamma\in H^{(d;m)}_{n}(\mathcal{A})$ such that $A=B\Gamma$. 

\begin{Def}
    We define the \textit{nc Grassmannian} $Gr^{(d;m)}(\mathcal{A})$ over $\mathcal{A}$ by
    \[
    Gr^{(d;m)}(\mathcal{A}):=\bigsqcup_{n\geq1}Gr^{(d;m)}_{n}(\mathcal{A}) \mbox{ with } Gr^{(d;m)}_{n}(\mathcal{A}):=GL_{m}(M_{n}(\mathcal{A}))/\sim^{(d;m)}_{n}.
    \]
    In the case of $(d;m)=(1;2)$, we simply write $Gr(\mathcal{A})$ instead of $Gr^{(1;2)}(\mathcal{A})$.
\end{Def}

\begin{Rem}
    \begin{enumerate}
        \item For any $A,B\in GL_{m}(M_{n}(\mathcal{A}))$, it follows that $A\sim^{(d;m)}_{n}B$ if and only if there exists an element $Z\in H^{(d;m)}_{n}(\mathcal{A})$ such that $A
        \left[\begin{smallmatrix}
            0\\
            I_{n}
        \end{smallmatrix}\right]
        =B
        \left[\begin{smallmatrix}
            0\\
            Z
        \end{smallmatrix}\right]$.
        \item If $\mathcal{R}=\mathbb{C}$ and $\mathcal{A}$ is a Banach algebra over $\mathbb{C}$ with $(d;m)=(1;2)$, then $Gr(\mathcal{A})$ is nothing but Voiculescu's non-commutative Riemann sphere \cite{v10}.
    \end{enumerate}
\end{Rem}

The affine space $M(\mathcal{A})^{d(m-d)}$ is embedded into $Gr^{(d;m)}(\mathcal{A})$ by
\begin{equation}
M_{n}(\mathcal{A})^{d(m-d)}\simeq M_{m-d,d}(M_{n}(\mathcal{A}))\ni
X
\mapsto
\widetilde{X}:=
\left[\begin{smallmatrix}
    0&I_{d}\otimes I_{n}\\
    I_{m-d}\otimes I_{n}&X
\end{smallmatrix}\right]/\sim^{(d;m)}_{n}\in Gr^{(d;m)}_{n}(\mathcal{A}).
\end{equation}
Let $\Omega$ be an nc set over $M(\mathcal{A})$ (in the sense of \cite{kvv14}), then it is easy to see that $\widetilde{\Omega}=\{\widetilde{X}\,|\,X\in\Omega\}$ is an nc set over $Gr^{(d;m)}(\mathcal{A})$. Let $\Omega$ be a (non-necessarily nc) set over $Gr^{(d;m)}(\mathcal{A})$. If $\Omega\cap \widetilde{M(\mathcal{A})}$ is nc over $Gr^{(d;m)}(\mathcal{A})$, then the inverse image of $\Omega\cap \widetilde{M(\mathcal{A})}$ by the above embedding is nc over $M(\mathcal{A})$. Hence, we identify $\Omega\cap \widetilde{M(\mathcal{A})}$ with its inverse image by the embedding and use the same symbol for them.  

\section{Nc sets and nc functions over $Gr^{(d;m)}(\mathcal{A})$}\label{section_ncsset_function}
First, we will define two operations for elements of $Gr^{(d;m)}(\mathcal{A})$.
\begin{Def}\label{def_nc_operation}
    For two elements $A\in M_{m}(M_{n}(\mathcal{A}))$ and $A'\in M_{m}(M_{n'}(\mathcal{A}))$, we define an element $A\widetilde{\oplus}A'$ of $M_{m}(M_{n+n'}(\mathcal{A}))$ by $A\widetilde{\oplus}A'=\sum_{i,j=1}^{m}(A_{i,j}\oplus A'_{i,j})\otimes e_{i,j}^{(m)}$, where $A=\sum_{i,j=1}^{m}A_{i,j}\otimes e^{(m)}_{i,j}$ and $A'=\sum_{i,j=1}^{m}A'_{i,j}\otimes e^{(m)}_{i,j}$ with $A_{i,j}\in M_{n}(\mathcal{A})$, $A'_{i,j}\in M_{n'}(\mathcal{A})$. (Here, $e^{(m)}_{i,j}$ denotes the $m$ by $m$ matrix unit of $(i,j)$-entry.)

    For two elements $\sigma\in Gr^{(d;m)}_{n}(\mathcal{A})$ and $\sigma'\in Gr^{(d;m)}_{n'}(\mathcal{A})$, we define the \textit{direct sum} $\sigma\oplus\sigma'\in Gr^{(d;m)}_{n+n'}(\mathcal{A})$ by $\sigma\oplus\sigma'=(A\widetilde{\oplus}A')/\sim^{(d;m)}_{n+n'}$ if $A\in\sigma$ and $A'\in\sigma'$.

    For an element $\sigma\in Gr^{(d;m)}_{n}(\mathcal{A})$ and $g\in GL_{m}(M_{n}(\mathcal{A}))$, we define $g\sigma:=(gA)/\sim^{(d;m)}_{n}$ for some $A\in\sigma$. For any $s\in GL_{n}(\mathcal{R})$, we define the \textit{similarity} $s\cdot \sigma\in Gr^{(d;m)}_{n}(\mathcal{A})$ by $s\cdot\sigma=s^{\oplus m}\sigma$.
\end{Def}

It is easy to see that the above two operations are well defined (i.e., the definitions are independent of the choice of representatives $A,A'$ of $\sigma,\sigma'$, respectively). The following are central objects in this paper:

\begin{Def}\label{def_nc_subset_function}
    An object $\Omega=\bigsqcup_{n\geq1}\Omega_{n}$ is called an \textit{nc set} over $Gr^{(d;m)}(\mathcal{A})$ if $\Omega$ satisfies the following conditions:
    \begin{enumerate}
        \item Each $\Omega_{n}$ is a subset of $Gr^{(d;m)}_{n}(\mathcal{A})$.
        \item For any $n,n'\in\mathbb{Z}_{\geq1}$, we have $\Omega_{n}\oplus\Omega_{n'}\subset\Omega_{n+n'}$.
    \end{enumerate}

    Let $\mathcal{B}$ be an $\mathcal{R}$-module and $\Omega$ be an nc set over $Gr^{(d;m)}(\mathcal{A})$. Then, a function $f:\Omega\to M(\mathcal{B})$ is \textit{graded} if $f(\Omega_{n})\subset M_{n}(\mathcal{B})$ for any $n\in\mathbb{Z}_{\geq1}$. A graded function $f:\Omega\to M(\mathcal{B})$ is \textit{nc} if $f$ satisfies the following conditions:
    \begin{enumerate}
        \item[(D)] \textit{Direct-sum-preserving}: $f(\sigma\oplus\sigma')=f(\sigma)\oplus f(\sigma')$ for any $\sigma,\sigma'\in \Omega$.
        \item[(S)] \textit{Similarity-preserving}: $f(s\cdot \sigma)=\mathrm{Ad}(s)(f(\sigma)):=sf(\sigma)s^{-1}$ whenever $\sigma\in\Omega_{n}$ and $s\in GL_{n}(\mathcal{R})$ satisfy $s\cdot\sigma\in\Omega_{n}$.
    \end{enumerate}
    We denote by $\mathcal{T}^{(0)}(\Omega;\mathcal{B})$ the space of all nc functions from $\Omega$ to $M(\mathcal{B})$.
\end{Def}


The following is the Grassmannian extension of the intertwining property \cite[Proposition 2.1]{kvv14}:

\begin{Prop}\label{prop_grassmann_intertwining}
    Let $\Omega$ be an nc set over $Gr^{(d;m)}(\mathcal{A})$, and let $f$ be a graded function from $\Omega$ to $M(\mathcal{B})$. Then, the following are equivalent:
    \begin{enumerate}
        \item[(a)] $f$ is nc.
        \item[(b)] $f$ satisfies the following condition:
        \begin{enumerate}
            \item[(I)] For any $n,n'\in\mathbb{Z}_{\geq1}$, $T\in M_{n,n'}(\mathcal{R})$ and $\sigma\in \Omega_{n}$, $\sigma'\in\Omega_{n'}$, we have
        \[
        \left[\begin{smallmatrix}
            I_{n}&T\\
            0&I_{n'}
        \end{smallmatrix}\right]\cdot(\sigma\oplus\sigma')=\sigma\oplus\sigma'\quad
        \Rightarrow\quad
        f(\sigma)T=Tf(\sigma').
        \]
        \end{enumerate}
    \end{enumerate}
\end{Prop}
\begin{proof}
    (a)$\Rightarrow$(b): By direct calculation.

    (b)$\Rightarrow$(a): We can check the following equalities:
    \begin{gather}
        \left[
    \begin{smallmatrix}
        I_{n+n'}&T_{1}\\
        0&I_{n}
    \end{smallmatrix}
    \right]\cdot((\sigma\oplus\sigma')\oplus\sigma)=(\sigma\oplus\sigma')\oplus\sigma,\quad T_{1}=\left[\begin{smallmatrix}
        I_{n}\\
        0_{n',n}
    \end{smallmatrix}\right],\label{eq_1}\\
    \left[
    \begin{smallmatrix}
        I_{n+n'}&T_{2}\\
        0&I_{n'}
    \end{smallmatrix}
    \right]\cdot((\sigma\oplus\sigma')\oplus\sigma')=(\sigma\oplus\sigma')\oplus\sigma',\quad T_{2}=\left[\begin{smallmatrix}
        0_{n,n'}\\
        I_{n'}
    \end{smallmatrix}\right],\label{eq_2}\\
    \left[\begin{smallmatrix}
        I_{n}&s^{-1}\\
        0&I_{n}
    \end{smallmatrix}\right]\cdot (\sigma\oplus(s\cdot\sigma))=\sigma\oplus(s\cdot\sigma)\label{eq_3}
    \end{gather}
    for any $\sigma\in\Omega_{n}$, $\sigma'\in\Omega_{n'}$ and $s\in GL_{n}(\mathcal{R})$.
    Using (b) with equalities (\ref{eq_1}),(\ref{eq_2}), we have
    \[
    f(\sigma\oplus\sigma')
    \left[\begin{smallmatrix}
        I_{n}\\
        0
    \end{smallmatrix}\right]
    =
    \left[\begin{smallmatrix}
        f(\sigma)\\
        0
    \end{smallmatrix}\right],\quad
    f(\sigma\oplus\sigma')
    \left[\begin{smallmatrix}
        0\\
        I_{n'}
    \end{smallmatrix}\right]
    =
    \left[\begin{smallmatrix}
        0\\
        f(\sigma')
    \end{smallmatrix}\right].
    \]
    Similarly, using (b) with equality (\ref{eq_3}), we obtain $f(\sigma)s^{-1}=s^{-1}f(s\cdot\sigma)$ as desired.
\end{proof}

\begin{Rem}
    In the case of $(d;m)=(1,2)$, consider \textit{affine} elements $\widetilde{X}=\left[\begin{smallmatrix}
        0&I_{n}\\
        I_{n}&X
    \end{smallmatrix}\right]/\sim_{n}\in Gr_{n}(\mathcal{A})$ and $\widetilde{X'}=\left[\begin{smallmatrix}
        0&I_{n'}\\
        I_{n'}&X'
    \end{smallmatrix}\right]/\sim_{n'}\in G_{n'}(\mathcal{A})$ with $X\in M_{n}(\mathcal{A})$ and $X'\in M_{n'}(\mathcal{A})$. For any $T\in M_{n,n'}(\mathcal{R})$, we can see that $\left[\begin{smallmatrix}
            I_{n}&T\\
            0&I_{n'}
        \end{smallmatrix}\right]\cdot(\widetilde{X}\oplus\widetilde{X'})=\widetilde{X}\oplus\widetilde{X'}$ if and only if $XT=TX'$. Thus, property (b) in Proposition \ref{prop_grassmann_intertwining} recovers the affine intertwining property \cite[Proposition 2.1]{kvv14}. 
\end{Rem}

\begin{Rem}(similarity-invariant envelope)\label{rem_similarity_envelope}
    Let $\Omega$ be an nc set over $Gr^{(d;m)}(\mathcal{A})$. We define the \textit{similarity-invariant envelope} $\widehat{\Omega}=\bigsqcup_{n\geq1}\widehat{\Omega}_{n}$ of $\Omega$ by $\widehat{\Omega}_{n}:=GL_{n}(\mathcal{R})\cdot\Omega_{n}:=\{s\cdot\sigma\,|\,s\in GL_{n}(\mathcal{R}),\sigma\in\Omega_{n}\}$. It is easy to see that $\widehat{\Omega}$ is an nc set over $Gr^{(d;m)}(\mathcal{A})$. Moreover, by construction, we have $s\cdot\widehat{\Omega}_{n}=\widehat{\Omega}_{n}$ for any $n\in\mathbb{Z}_{\geq1}$ and $s\in GL_{n}(\mathcal{R})$.
    
    Let $f$ be an nc function from $\Omega$ to $M(\mathcal{B})$. Then, there exists a unique nc function $\widehat{f}$ from $\widehat{\Omega}$ to $M(\mathcal{B})$ such that $\widehat{f}|_{\Omega}=f$.
\end{Rem}

The following is the multivariable version of an nc function.
\begin{Def}\label{def_multivariable_nc}
    Let $\Omega^{(0)},\Omega^{(1)}$ be nc sets over $Gr^{(d;m)}(\mathcal{A})$, and let $\mathcal{B}^{(0)},\mathcal{B}^{(1)}$ be $\mathcal{R}$-modules. Set $\mathrm{Hom}_{\mathcal{R}}(M(\mathcal{B}^{(1)}),M(\mathcal{B}^{(0)})):=\bigsqcup_{n,n'\geq1}\mathrm{Hom}_{\mathcal{R}}(M_{n,n'}(\mathcal{B}^{(1)}),M_{n,n'}(\mathcal{B}^{(0)}))$. A function $f:\Omega^{(0)}\times\Omega^{(1)}\to \mathrm{Hom}_{\mathcal{R}}(M(\mathcal{B}^{(1)}),M(\mathcal{B}^{(0)}))$ is \textit{graded} if $f(\Omega^{(0)}_{n};\Omega^{(1)}_{n'})\subset\mathrm{Hom}_{\mathcal{R}}(M_{n,n'}(\mathcal{B}^{(1)}),M_{n,n'}(\mathcal{B}^{(0)}))$ for any $n,n'\in\mathbb{Z}_{\geq1}$.
    A graded function $f:\Omega^{(0)}\times\Omega^{(1)}\to \mathrm{Hom}_{\mathcal{R}}(M(\mathcal{B}^{(1)}),M(\mathcal{B}^{(0)}))$ is \textit{nc of order $1$} if $f$ satisfies the following:
    \begin{enumerate}
        \item \textit{Direct-sum-preserving}:
        \begin{enumerate}
            \item[(D-0)] $f(\sigma_{0}\oplus\sigma_{0}';\sigma_{1})\left(\left[\begin{smallmatrix}
                X\\
                X'
            \end{smallmatrix}\right]\right)
            =\left[\begin{smallmatrix}
                f(\sigma_{0};\sigma_{1})(X)\\
                f(\sigma'_{0};\sigma_{1})(X')
            \end{smallmatrix}\right]$ for any $\sigma_{0}\in\Omega^{(0)}_{n_{0}}$, $\sigma_{0}'\in\Omega_{n'_{0}}^{(0)}$, $\sigma_{1}\in\Omega_{n_{1}}^{(1)}$ and $X\in M_{n_{0},n_{1}}(\mathcal{A})$, $X'\in M_{n_{0}',n_{1}}(\mathcal{A})$\label{D_0}.
            \item[(D-1)] $f(\sigma_{0};\sigma_{1}\oplus\sigma_{1}')\left(\left[\begin{smallmatrix}
                Y&Y'\\
            \end{smallmatrix}\right]\right)
            =\left[\begin{smallmatrix}
                f(\sigma_{0};\sigma_{1})(Y)&f(\sigma_{0};\sigma_{1}')(Y')\\
            \end{smallmatrix}\right]$ for any $\sigma_{0}\in\Omega^{(0)}_{n_{0}}$, $\sigma_{1}\in\Omega_{n_{1}}^{(1)}$, $\sigma_{1}'\in\Omega_{n'_{1}}^{(1)}$ and $Y\in M_{n_{0},n_{1}}(\mathcal{A})$, $Y'\in M_{n_{0},n'_{1}}(\mathcal{A})$\label{D_1}.
        \end{enumerate}
        \item \textit{Similarity-preserving}:
        \begin{enumerate}
            \item[(S)]
            $f(s_{0}\cdot\sigma_{0};s_{1}\cdot\sigma_{1})(X)=s_{0}f(\sigma_{0};\sigma_{1})(s_{0}^{-1}Xs_{1})s_{1}^{-1}$ for any $s_{0}\in GL_{n_{0}}(\mathcal{R})$, $s_{1}\in GL_{n_{1}}(\mathcal{R})$, $\sigma_{0}\in\Omega^{(0)}_{n_{0}}$, $\sigma_{1}\in\Omega^{(1)}_{n_{1}}$ and $X\in M_{n_{0},n_{1}}(\mathcal{A})$ with $s_{0}\cdot\sigma_{0}\in\Omega^{(0)}_{n_{0}}$ and $s_{1}\cdot\sigma_{1}\in\Omega^{(1)}_{n_{1}}$\label{S_0}.
        \end{enumerate}
    \end{enumerate}
    We denote by $\mathcal{T}^{(1)}(\Omega^{(0)},\Omega^{(1)};\mathcal{B}^{(0)},\mathcal{B}^{(1)})$ the space of all nc functions of order $1$ from $\Omega^{(0)}\times\Omega^{(1)}$ to $\mathrm{Hom}_{\mathcal{R}}(M(\mathcal{B}^{(0)}),M(\mathcal{B}^{(1)}))$.
\end{Def}

\begin{Rem}
    We can also show a multivariable counterpart of Proposition \ref{prop_grassmann_intertwining}. However, we omit it here, since it will not be used later in this paper. Moreover, we can formulate the notion of nc functions of order higher than $2$.
\end{Rem}

\section{The nc $(u,v)$-difference-differential operator $\widetilde{\Delta}_{u,v}^{(d;m)}$}\label{section_ddoperator}
In this section, we will construct non-commutative differential operators for nc functions over the nc Grassmannian.
Our strategy is to adapt the argument of \cite{kvv14} to the framework of the nc Grassmannian.
This will be achieved by replacing the affine intertwining property $XT=TY\,\Rightarrow\, f(X)T=Tf(Y)$ used in \cite{kvv14} with the Grassmannian intertwining property $\left[\begin{smallmatrix}
    I_{n}&T\\
    0&I_{n'}
\end{smallmatrix}\right]\cdot(\sigma\oplus\sigma')=\sigma\oplus\sigma'\,\Rightarrow\,f(\sigma)T=Tf(\sigma')$ (see Proposition \ref{prop_grassmann_intertwining}). 

\subsection{Construction}
We first introduce new notation to simplify of our exposition. Fix $(u,v)\in[m-d]\times[d]$, $\sigma\in Gr^{(d;m)}_{n}(\mathcal{A})$, $\sigma'\in Gr^{(d;m)}_{n'}(\mathcal{A})$ and $X\in M_{n,n'}(\mathcal{A})$. We define an element $(A;A')_{u,v}(X)/\sim^{(d;m)}_{n+n'}$ of $Gr^{(d;m)}_{n+n'}(\mathcal{A})$ by
\begin{equation}
        \begin{aligned}
            (A;A')_{u,v}(X):=\left(A\widetilde{\oplus}A'+\sum_{j=1}^{d}
            \left[\begin{smallmatrix}
                 0_{n}&XA'_{v,m-d+j}\\
                 0_{n',n}&0_{n'}
            \end{smallmatrix}\right]
            \otimes e^{(m)}_{d+u,m-d+j}\right) \label{eq_6}
        \end{aligned}
\end{equation}
for any $A\in\sigma$ and $A'\in\sigma'$. Note that the equivalence class $(A;A')_{u,v}(X)/\sim^{(d;m)}_{n+n'}$ is independent of the choice of $A\in\sigma$ and $A'\in\sigma'$.

\begin{Nota}
    We denote by $(\sigma;\sigma')_{u,v}(X)$ the element $(A;A')_{u,v}(X)/\sim^{(d;m)}_{n+n'}$ with $A\in\sigma$ and $A'\in\sigma'$.
\end{Nota}

\begin{Def}
    Let $\Omega$ be an nc set over $Gr^{(d;m)}(\mathcal{A})$, and fix $(u,v)\in[m-d]\times[d]$. We say that $\Omega$ is \textit{$(u,v)$-admissible} if for any $\sigma\in\Omega_{n}$, $\sigma'\in\Omega_{n'}$ and any $X\in M_{n,n'}(\mathcal{A})$, there exists an invertible element $r\in\mathcal{R}$ such that $(\sigma;\sigma')_{u,v}(rX)\in\Omega_{n+n'}$. In the case of $(d;m)=(1;2)$, we simply say that $\Omega$ is \textit{admissible} instead of $(1,1)$-admissible.
\end{Def}

\begin{Lem}
    Let $\Omega$ be an nc set over $Gr^{(d;m)}(\mathcal{A})$, and let $f$ be an nc function from $\Omega$ to $M(\mathcal{B})$. Fix $(u,v)\in[m-d]\times[d]$. Suppose that $\sigma\in\Omega_{n}$, $\sigma'\in\Omega_{n'}$ and $X\in M_{n,n'}(\mathcal{A})$ satisfy $(\sigma;\sigma')_{u,v}(X)\in\Omega_{n+n'}$. Then, there exists a unique element $H\in M_{n,n'}(\mathcal{B})$, denoted by $(\widetilde{\Delta}^{(d;m)}_{u,v}f)(\sigma;\sigma')(X)$, such that 
    $f\left((\sigma;\sigma')_{u,v}(X)\right)=\left[\begin{smallmatrix}
        f(\sigma)&H\\
        0&f(\sigma')
    \end{smallmatrix}\right]$. Moreover, if $(\sigma;\sigma')_{u,v}(rX)\in\Omega_{n+n'}$ for some $r\in\mathcal{R}$, then we have $(\widetilde{\Delta}^{(d;m)}_{u,v}f)(\sigma;\sigma')(rX)=r(\widetilde{\Delta}^{(d;m)}_{u,v}f)(\sigma;\sigma')(X)$.
\end{Lem}
\begin{proof}
    Combine Proposition \ref{prop_grassmann_intertwining} with the following equalities:
    \begin{gather}
        \left[
    \begin{smallmatrix}
        I_{n+n'}&T_{1}\\
        0&I_{n}
    \end{smallmatrix}
    \right]\cdot((\sigma;\sigma')_{u,v}(X)\oplus\sigma)=(\sigma;\sigma')_{u,v}(X)\oplus\sigma,\quad T_{1}=\left[\begin{smallmatrix}
        I_{n}\\
        0
    \end{smallmatrix}\right],\label{eq_7}\\
    \left[
    \begin{smallmatrix}
        I_{n'}&T_{2}\\
        0&I_{n+n'}
    \end{smallmatrix}
    \right]\cdot(\sigma'\oplus(\sigma;\sigma')_{u,v}(X))=\sigma'\oplus(\sigma;\sigma')_{u,v}(X),\quad T_{2}=\left[
    \begin{smallmatrix}
        0&I_{n'}
    \end{smallmatrix}
    \right],\label{eq_8}\\
    \left[
    \begin{smallmatrix}
        I_{n+n'}&rI_{n}\oplus I_{n'}\\
        0&I_{n+n'}
    \end{smallmatrix}
    \right]\cdot((\sigma;\sigma')_{u,v}(X)\oplus(\sigma;\sigma')_{u,v}(X))=(\sigma;\sigma')_{u,v}(X)\oplus(\sigma;\sigma')_{u,v}(X)\label{eq_9}.
    \end{gather}
\end{proof}

\begin{Def}
    Let $u\in[m-d]$ and $v\in[d]$, and let $\Omega$ be a $(u,v)$-admissible nc set over $Gr^{(d;m)}(\mathcal{A})$. Let $f$ be an nc function from $\Omega$ to $M(\mathcal{B})$. For any $\sigma\in\Omega_{n}$, $\sigma'\in\Omega_{n'}$ and $X\in M_{n,n'}(\mathcal{A})$, we define
    \[
    (\widetilde{\Delta}^{(d;m)}_{u,v}f)(\sigma;\sigma')(X):=r^{-1}(\widetilde{\Delta}^{(d;m)}_{u,v}f)(\sigma;\sigma')(rX),
    \]
    where $r\in\mathcal{R}$ is an invertible element such that $(\sigma;\sigma')_{u,v}(rX)\in\Omega_{n+n'}$. In the case of $(d;m)=(1;2)$, we simply write $\widetilde{\Delta}$ instead of $\widetilde{\Delta}^{(1;2)}_{1,1}$.
\end{Def}

It is easy to see that $(\widetilde{\Delta}^{(d;m)}_{u,v}f)(\sigma;\sigma')(X)$ is well defined, that is, it is independent of the choice of invertible elements $r\in\mathcal{R}$ such that $(\sigma;\sigma')_{u,v}(rX)\in\Omega_{n+n'}$. Hence, $(\widetilde{\Delta}^{(d;m)}_{u,v}f)(\sigma;\sigma')$ defines a well-defined map from $M_{n,n'}(\mathcal{A})$ to $M_{n,n'}(\mathcal{B})$ for each $n,n'\in\mathbb{Z}_{\geq1}$ and $\sigma\in\Omega_{n}$, $\sigma'\in\Omega_{n'}$.

\subsection{Properties}
In the sequel, we will see that $\widetilde{\Delta}^{(d;m)}_{u,v}$ is an $\mathcal{R}$-linear map from $\mathcal{T}^{(0)}(\Omega;\mathcal{B})$ to $\mathcal{T}^{(1)}(\Omega,\Omega;\mathcal{A},\mathcal{B})$. Here is one remark.

\begin{Rem}\label{rem_ddop_similari_enve}
    Let $\Omega$ be a $(u,v)$-admissible nc set over $Gr^{(d;m)}(\mathcal{A})$, and let $f$ be an nc function from $\Omega$ to $M(\mathcal{B})$. Then, it is easy to see that the similarity-invariant envelope $\widehat{\Omega}$ is also $(u,v)$-admissible. In particular, we have 
    $(\sigma;\sigma')_{u,v}(X)\in\widehat{\Omega}_{n+n'}$ 
    for any $\sigma\in\widehat{\Omega}_{n}$, $\sigma'\in\widehat{\Omega}_{n'}$ and $X\in M_{n,n'}(\mathcal{A})$. 
    Moreover, $(\widetilde{\Delta}^{(d;m)}_{u,v}\widehat{f})(\sigma;\sigma')(X)=(\widetilde{\Delta}^{(d;m)}_{u,v}f)(\sigma;\sigma')(X)$ for any $\sigma\in\Omega_{n}$, $\sigma'\in\Omega_{n'}$ and $X\in M_{n,n'}(\mathcal{A})$.

    Hence, it suffices to prove the statements in this section under the assumption that $\widehat{\Omega}=\Omega$. 
\end{Rem}

\begin{Lem}\label{lem_directsaumpreserving_ddop}
    Let $\Omega$ be a $(u,v)$-admissible nc set over $Gr^{(d;m)}(\mathcal{A})$ and $f$ an nc function from $\Omega$ to $M(\mathcal{B})$. Then, $(\widetilde{\Delta}^{(d;m)}_{u,v}f)(\sigma;\sigma')$ satisfies properties (D-0),(D-1) in Definition \ref{def_multivariable_nc}.
\end{Lem}
\begin{proof}
    Combine Proposition \ref{prop_grassmann_intertwining} with the following equalities:
    \begin{equation}
        \begin{split}
            \begin{aligned}
                \,&
                \left[
                \begin{smallmatrix}
                    I_{n_{0}+n_{1}+n_{1}'}&T_{1}\\
                    0&I_{n_{0}+n_{1}}
                \end{smallmatrix}
                \right]
                \cdot\left((\sigma_{0};\sigma_{1}\oplus\sigma_{1}')_{u,v}
                (\left[\begin{smallmatrix}
                Y&Y'
                \end{smallmatrix}\right])\oplus(\sigma_{0};\sigma_{1})_{u,v}(Y)\right)=\\
                &\hspace{3cm}=(\sigma_{0};\sigma_{1}\oplus\sigma_{1}')_{u,v}
                (\left[\begin{smallmatrix}
                Y&Y'
                \end{smallmatrix}\right])\oplus(\sigma_{0};\sigma_{1})_{u,v}(Y),\quad T_{1}=
                \left[\begin{smallmatrix}
                        I_{n_{0}}&0\\
                        0&I_{n_{1}}\\
                        0&0
                    \end{smallmatrix}\right],
            \end{aligned}
        \end{split}
    \end{equation}
    \begin{equation}
        \begin{split}
            \begin{aligned}
                \,&
                \left[
                \begin{smallmatrix}
                    I_{n_{0}+n_{1}+n_{1}'} & T_{2} \\
                    0 & I_{n_{0}+n_{1}'}
                \end{smallmatrix}
                \right]\cdot\left((\sigma_{0};\sigma_{1}\oplus\sigma_{1}')_{u,v}(\left[\begin{smallmatrix}
                    Y&Y'
                \end{smallmatrix}\right])\oplus(\sigma_{0};\sigma_{1}')_{u,v}(Y')\right)=\\
                &\hspace{3cm}=(\sigma_{0};\sigma_{1}\oplus\sigma_{1}')_{u,v}(\left[\begin{smallmatrix}
                    Y&Y'
                \end{smallmatrix}\right])\oplus(\sigma_{0};\sigma_{1}')_{u,v}(Y'),\quad
                T_{2}=\left[
                \begin{smallmatrix}
                        I_{n_{0}}&0\\
                        0&0\\
                        0&I_{n_{1}'}
                    \end{smallmatrix}
                \right],
            \end{aligned}
        \end{split}
    \end{equation}
    \begin{equation}
        \begin{split}
            \begin{aligned}
                \,&
                \left[
                \begin{smallmatrix}
                    I_{n_{0}+n_{1}} & T_{3}\\
                    0 & I_{n_{0}+n_{0}'+n_{1}}
                \end{smallmatrix}
                \right]\cdot\left((\sigma_{0};\sigma_{1})_{u,v}(X)\oplus(\sigma_{0}\oplus\sigma_{0}';\sigma_{1})_{u,v}
                \left(\left[\begin{smallmatrix}
                    X\\
                    X'
                \end{smallmatrix}\right]\right)\right)=\\
                &\hspace{3cm}=(\sigma_{0};\sigma_{1})_{u,v}(X)\oplus(\sigma_{0}\oplus\sigma_{0}';\sigma_{1})_{u,v}
                \left(\left[\begin{smallmatrix}
                    X\\
                    X'
                \end{smallmatrix}\right]\right),\quad 
                T_{3}=\left[
                \begin{smallmatrix}
                        I_{n_{0}}&0&0\\
                        0&0&I_{n_{1}}
                    \end{smallmatrix}
                \right],
            \end{aligned}
        \end{split}
    \end{equation}
    \begin{equation}
        \begin{split}
            \begin{aligned}
                \,&
                \left[
                \begin{smallmatrix}
                    I_{n'_{0}+n_{1}} & T_{4}\\
                    0 & I_{n_{0}+n_{0}'+n_{1}}
                \end{smallmatrix}
                \right]\cdot\left((\sigma'_{0};\sigma_{1})_{u,v}(X')\oplus(\sigma_{0}\oplus\sigma_{0}';\sigma_{1})_{u,v}
                \left(\left[\begin{smallmatrix}
                    X\\
                    X'
                \end{smallmatrix}\right]\right)\right)=\\
                &\hspace{3cm}=(\sigma_{0}';\sigma_{1})_{u,v}(X')\oplus(\sigma_{0}\oplus\sigma_{0}';\sigma_{1})_{u,v}
                \left(\left[\begin{smallmatrix}
                    X\\
                    X'
                \end{smallmatrix}\right]\right),\quad
                T_{4}=\left[
                \begin{smallmatrix}
                        0&I_{n_{0}'}&0\\
                        0&0&I_{n_{1}}
                    \end{smallmatrix}
                \right]
            \end{aligned}
        \end{split}
    \end{equation}
    for any $\sigma_{0}\in\Omega_{n_{0}}$, $\sigma_{0}'\in\Omega_{n_{0}'}$, $\sigma_{1}\in\Omega_{n_{1}}$, $\sigma_{1}'\in\Omega_{n_{1}'}$ and $X\in M_{n_{0},n_{1}}(\mathcal{A})$, $X'\in M_{n'_{0},n_{1}}(\mathcal{A})$, $Y\in M_{n_{0},n_{1}}(\mathcal{A})$, $Y'\in M_{n_{0},n'_{1}}(\mathcal{A})$.
\end{proof}

From Lemma \ref{lem_directsaumpreserving_ddop}, we can see the following:
\begin{Prop}
    Let $\Omega$ be a $(u,v)$-admissible nc set over $Gr^{(d;m)}(\mathcal{A})$, and let $f$ be an nc function from $\Omega$ to $M(\mathcal{B})$. 
    Then, $\widetilde{\Delta}^{(d;m)}_{u,v}$ is an $\mathcal{R}$-linear map from $\mathcal{T}^{(0)}(\Omega;\mathcal{B})$ to $\mathcal{T}^{(1)}(\Omega,\Omega;\mathcal{A},\mathcal{B})$.
\end{Prop}
\begin{proof}
    The $\mathcal{R}$-homogeneity of $(\widetilde{\Delta}^{(d;m)}_{u,v}f)(\sigma;\sigma')$ can be shown by the same computation as in \cite[Proposition 2.4]{kvv14}. 
    
    We can also show the additivity, that is, 
    \[
    (\widetilde{\Delta}^{(d;m)}f)(\sigma;\sigma')(X+Y)=(\widetilde{\Delta}^{(d;m)}f)(\sigma;\sigma')(X)+(\widetilde{\Delta}^{(d;m)}f)(\sigma;\sigma')(Y)
    \] 
    for any $\sigma\in\Omega_{n}$, $\sigma'\in\Omega_{n'}$ and $X,Y\in M_{n,n'}(\mathcal{A})$, by combining Proposition \ref{prop_grassmann_intertwining} with the following equality
    \begin{equation}
        \begin{split}
            \begin{aligned}
                \,&\left[
                \begin{smallmatrix}
                    I_{n+n'} & T_{1} \\
                    0 & I_{2n+n'} 
                \end{smallmatrix}
                \right]\cdot\left((\sigma;\sigma')_{u,v}(X+Y)\oplus
                (\sigma\oplus\sigma;\sigma')_{u,v}\left(
                \left[\begin{smallmatrix}
                    X\\
                    Y
                \end{smallmatrix}\right]\right)
                \right)=\\
                &\hspace{2cm}=(\sigma;\sigma')_{u,v}(X+Y)\oplus
                (\sigma\oplus\sigma;\sigma')_{u,v}\left(
                \left[\begin{smallmatrix}
                    X\\
                    Y
                \end{smallmatrix}\right]\right),\quad
                T_{1}=\left[
                \begin{smallmatrix}
                        I_{n}&I_{n}&0\\
                        0&0&I_{n'}
                    \end{smallmatrix}
                \right]
            \end{aligned}
        \end{split}
    \end{equation}
    and, after that, using Lemma \ref{lem_directsaumpreserving_ddop}. 

    It remains to show that $(\widetilde{\Delta}^{(d;m)}_{u,v}f)(\sigma;\sigma')$ satisfies property (S) in Definition \ref{def_multivariable_nc}. This follows by using Proposition \ref{prop_grassmann_intertwining} with the following equality
    \begin{equation}
        \begin{split}
            \begin{aligned}
                \,&\left[
                \begin{smallmatrix}
                    I_{n+n'} & T_{2} \\
                    0 & I_{n+n'}
                \end{smallmatrix}
                \right]\cdot\left((s\cdot\sigma;s'\cdot\sigma')_{u,v}(X)\oplus(\sigma;\sigma')_{u,v}(s^{-1}Xs')\right)=\\
                &\hspace{2cm}=(s\cdot\sigma;s'\cdot\sigma')_{u,v}(X)\oplus(\sigma;\sigma')_{u,v}(s^{-1}Xs'),\quad 
                T_{2}=\left[
                \begin{smallmatrix}
                        s&0\\
                        0&s'
                    \end{smallmatrix}
                \right].
            \end{aligned}
        \end{split}
    \end{equation}
\end{proof}

For each $(u,v)\in[m-d]\times[d]$, we call $\widetilde{\Delta}^{(d;m)}_{u,v}$ the \textit{nc $(u,v)$-difference-differential operator}.

\subsection{A generalization of the first order difference formulae}

For any $\sigma\in\Omega_{n}$ and $Y\in M_{m-d,d}(M_{n}(\mathcal{B}))$, we define 
\[\sigma(Y):=
\left[
\begin{smallmatrix}
    I_{d}\otimes I_{n}&0\\
    -Y&I_{m-d}\otimes I_{n}
\end{smallmatrix}
\right]\sigma.
\]
The following is a generalization of \cite[Theorem 2.10]{kvv14} to the present framework of the nc $(d;m)$-Grassmannian.
\begin{Prop}\label{prop_expansion}
    Let $\Omega$ be an nc set over $Gr^{(d;m)}(\mathcal{A})$, $f$ an nc function from $\Omega$ to $M(\mathcal{B})$ and $\sigma\in\Omega_{n}$. Then, we have
    \[
    f(\sigma)-f\left(\sigma(e^{(m-d,d)}_{u,v}\otimes X)\right)=(\widetilde{\Delta}^{(d;m)}_{u,v}f)\left(\sigma;\sigma(e^{(m-d,d)}_{u,v}\otimes X)\right)(X)
    \]
    for any $X\in M_{n}(\mathcal{A})$ with $\sigma(e^{(m-d,d)}_{u,v}\otimes X)\in\Omega_{n}$ and $((\sigma;\sigma(e^{(m-d,d)}_{u,v}\otimes X))_{u,v}(X)\in \Omega_{2n}$ with $(u,v)\in [m-d]\times[d]$.
\end{Prop}
\begin{proof}
    Combine Proposition \ref{prop_grassmann_intertwining} with the following equality
    \begin{align*}
    \,&\left[
    \begin{smallmatrix}
        I_{2n}&T
        \\
        0&I_{n}
    \end{smallmatrix}
    \right]\cdot\left(((\sigma;\sigma(e^{(m-d,d)}_{u,v}\otimes X))_{u,v}(X)\oplus
    \sigma(e^{(m-d,d)}_{u,v}\otimes X)\right)=\\
    &\hspace{2cm}=((\sigma;\sigma(e^{(m-d,d)}_{u,v}\otimes X))_{u,v}(X)\oplus
    \sigma(e^{(m-d,d)}_{u,v}\otimes X),\quad T=\left[\begin{smallmatrix}
            I_{n}\\
            -I_{n}
        \end{smallmatrix}
        \right].
    \end{align*}
\end{proof}

\section{The $(d;m)$-Grassmannian $\mathcal{B}$-resolvent $\widetilde{\mathfrak{R}}^{(d;m)}(\pi;\mathcal{B}|v,u)$ with respect to $(u,v)$}\label{section_Grassmann_resolvent}\label{section_resolvent}

In this section, we will generalize Voiculescu's resolvent set and resolvent function to those over the nc $(d;m)$-Grassmannian $Gr^{(d;m)}(\mathcal{A})$. 
Let $1\in\mathcal{B}\subset\mathcal{A}$ be an inclusion of unital $\mathcal{R}$-algebras.
For any $A,B\in M_{m}(M_{n}(\mathcal{A}))$, we set
\[
r^{(d;m)}(A;B):=
\left[
\begin{smallmatrix}
    A_{1,d+1}&\cdots&A_{1,m}&B_{1,m-d+1}&\cdots&B_{1,m}\\
    \vdots&\ddots&\vdots&\vdots&\ddots&\vdots\\
    A_{m,d+1}&\cdots&A_{m,m}&B_{m,m-d+1}&\cdots&B_{m,m}
\end{smallmatrix}
\right].
\]

\begin{Def}
    Let $\pi\in Gr^{(m-d;m)}_{n}(\mathcal{A})$ and $\sigma\in Gr^{(d;m)}_{n}(\mathcal{A})$. We say that $\pi$ and $\sigma$ are \textit{transversal in $\mathcal{B}$} if $r^{(d;m)}(A;B)$ is invertible in $M_{m}(M_{n}(\mathcal{B}))$ for some $A\in\pi$ and $B\in\sigma$.
\end{Def}

This is a generalization of one introduced in \cite[subsection 4.5]{v10}. By the definition of the equivalence relations $\sim^{(d;m)}_{n}$ and $\sim^{(m-d;m)}_{n}$, it follows that the transversality of $\pi$ and $\sigma$ in $\mathcal{B}$ is independent of the choice of $A\in\pi$ and $B\in\sigma$. That is, if there exist $A\in\pi$ and $B\in\sigma$ such that $r^{(d;m)}(A;B)$ is invertible in $M_{m}(M_{n}(\mathcal{B}))$, then so is $r^{(d;m)}(A';B')$ for any $A'\in\pi$ and $B'\in\sigma$. 

\begin{Def}
    For an element $\pi\in Gr^{(m-d;m)}_{1}(\mathcal{A})$, we define $\widetilde{\rho}^{(d;m)}(\pi;\mathcal{B}):=\bigsqcup_{n\geq1}\widetilde{\rho}^{(d;m)}_{n}(\pi;\mathcal{B})$, where
    \[
    \widetilde{\rho}^{(d;m)}_{n}(\pi;\mathcal{B})
    =\left\{
    \sigma\in Gr^{(d;m)}_{n}(\mathcal{B})\,\middle|\,\pi^{\oplus n}:=\underbrace{\pi\oplus\cdots\oplus\pi}_{n}\mbox{ and }\sigma\mbox{ are transversal in }\mathcal{A}
    \right\}.
    \]
    We call $\widetilde{\rho}^{(d;m)}_{n}(\pi;\mathcal{B})$ and  $\widetilde{\rho}^{(d;m)}(\pi;\mathcal{B})$ the \textit{$(d;m)$-Grassmannian $n$-th $\mathcal{B}$-resolvent set} and the \textit{$(d;m)$-Grassmannian full $\mathcal{B}$-resolvent set} of $\pi$, respectively. In the case of $(d;m)=(1;2)$, we will simply write $\widetilde{\rho}(\pi;\mathcal{B})$ instead of $\widetilde{\rho}^{(1;2)}(\pi;\mathcal{B})$.
\end{Def}

We can easily see that $\widetilde{\rho}^{(d;m)}(\pi;\mathcal{B})$ is not only an nc set over $Gr^{(d;m)}(\mathcal{B})$, but also enjoys the following stronger properties:
\begin{Prop}\label{lem_resolvent_fuuly_matricial}
    The $(d;m)$-Grassmannian full $\mathcal{B}$-resolvent set $\widetilde{\rho}^{(d;m)}(\pi;\mathcal{B})$ has the following properties:
    \begin{enumerate}
        \item For any $n\in\mathbb{Z}_{\geq1}$, $\widetilde{\rho}^{(d;m)}_{n}(\pi;\mathcal{B})\subset Gr^{(d;m)}_{n}(\mathcal{B})$.
        \item For any $n,n'\in\mathbb{Z}_{\geq1}$, $\widetilde{\rho}^{(d;m)}_{n}(\pi;\mathcal{B})\oplus\widetilde{\rho}^{(d;m)}_{n'}(\pi;\mathcal{B})=\widetilde{\rho}^{(d;m)}_{n+n'}(\pi;\mathcal{B})\cap \left(Gr^{(d;m)}_{n}(\mathcal{A})\oplus Gr^{(d;m)}_{n'}(\mathcal{A})\right)$.
        \item For any $n\in\mathbb{Z}_{\geq1}$ and $s\in GL_{n}(\mathcal{R})$, $s\cdot\widetilde{\rho}^{(d;m)}_{n}(\pi;\mathcal{B})=\widetilde{\rho}^{(d;m)}_{n}(\pi;\mathcal{B})$.
    \end{enumerate}
\end{Prop}

We are now in a position to define the resolvent function on $\widetilde{\rho}^{(d;m)}(\pi;\mathcal{B})$.
\begin{Def}
    Let $(u,v)\in[m-d]\times[d]$, and let $\pi$ be an element of $Gr^{(m-d;m)}_{1}(\mathcal{A})$. The \textit{$(d;m)$-Grassmannian $\mathcal{B}$-resolvent with respect to $(v,u)$} of $\pi$, denoted by $\widetilde{\mathfrak{R}}^{(d;m)}(\pi;\mathcal{B}|v,u)$, is the map from $\widetilde{\rho}^{(d;m)}(\pi;\mathcal{B})$ to $M(\mathcal{A})$ defined by 
    \[\widetilde{\mathfrak{R}}^{(d;m)}(\pi;\mathcal{B}|v,u)(\sigma)=[B_{v,m-d+1},B_{v,m-d+2},\dots,B_{v,m}]\times\zeta_{u}\] 
    for any $\sigma\in\widetilde{\rho}^{(d;m)}(\pi;\mathcal{B})$. Here, $a=[a_{i,j}]\in\pi$, $B=[B_{i,j}]\in\sigma$, where $a_{i,j}\in\mathcal{A}$ and $B_{i,j}\in M_{n}(\mathcal{A})$, and 
    \[
    r^{(d;m)}(a^{\widetilde{\oplus}n};B)^{-1}=
    \left[
    \begin{smallmatrix}
        \ast&\ast&\cdots&\ast\\
        \ast&\zeta_{1}&\cdots&\zeta_{m-d}
    \end{smallmatrix}
    \right],\quad \zeta_{u}\in M_{d,1}(M_{n}(\mathcal{A}))\,(u\in[m-d]).
    \]
    In the case of $(d;m)=(1;2)$, we simply write $\widetilde{\mathfrak{R}}(\pi;\mathcal{B})$ instead of $\widetilde{\mathfrak{R}}^{(1;2)}(\pi;\mathcal{B}|1,1)$.
\end{Def}

It is easy to see that the $(d;m)$-Grassmannian $\mathcal{B}$-resolvents are well defined. Moreover, we have the following:

\begin{Prop}
    Let $\pi$ be an element of $Gr_{1}^{(d;m)}(\mathcal{A})$. Then, $\widetilde{\mathfrak{R}}^{(d;m)}(\pi;\mathcal{B}|v,u)$ is an nc function from $\widetilde{\rho}^{(d;m)}(\pi;\mathcal{B})$ to $M(\mathcal{A})$.
\end{Prop}
\begin{proof}
    Fix $\sigma\in\widetilde{\rho}^{(d;m)}_{n}(\pi;\mathcal{B})$ and $\sigma'\in\widetilde{\rho}^{(d;m)}_{n'}(\pi;\mathcal{B})$. Let $\zeta_{1},\dots,\zeta_{m-d}\in M_{d,1}(M_{n}(\mathcal{A}))$ and $\zeta'_{1},\dots,\zeta'_{m-d}\in M_{d,1}(M_{n'}(\mathcal{A}))$ satisfy
    \[
    r^{(d;m)}(a^{\widetilde{\oplus}n};B)^{-1}=\left[\begin{smallmatrix}
        \ast&\ast&\cdots&\ast\\
        \ast&\zeta_{1}&\cdots&\zeta_{m-d}
    \end{smallmatrix}\right],\quad
    r^{(d;m)}(a^{\widetilde{\oplus}n'};B')^{-1}=\left[\begin{smallmatrix}
        \ast&\ast&\cdots&\ast\\
        \ast&\zeta'_{1}&\cdots&\zeta'_{m-d}
    \end{smallmatrix}\right]
    \]
    for some $a\in\pi$, $B\in\sigma$ and $B'\in\sigma'$.
    
    \textit{Direct-sum-preserving} (D): We have $r^{(d;m)}(a^{\widetilde{\oplus}(n+n')};B\widetilde{\oplus}B')=r^{(d;m)}(a^{\widetilde{\oplus}n};B)\widetilde{\oplus}r^{(d;m)}(a^{\widetilde{\oplus}n'};B')$. Hence,
    \begin{align*}
        \widetilde{\mathfrak{R}}^{(d;m)}\left(\pi;\mathcal{B}|v,u\right)(\sigma\oplus\sigma')
        &=\left[
        \begin{smallmatrix}
            B_{v,m-d+1}&0\\
            0&B_{v,m-d+1}'
        \end{smallmatrix},\dots,
        \begin{smallmatrix}
            B_{v,m}&0\\
            0&B_{v,m}'
        \end{smallmatrix}
        \right]\times(\zeta_{u}\widetilde{\oplus}\zeta_{u}')\\
        &=\widetilde{\mathfrak{R}}^{(d;m)}\left(\pi;\mathcal{B}|v,u\right)(\sigma)\oplus\widetilde{\mathfrak{R}}^{(d;m)}\left(\pi;\mathcal{B}|v,u\right)(\sigma')
    \end{align*}
    as desired.

    \textit{Similarity-preserving} (S): Let $s\in GL_{n}(\mathcal{R})$ with $s\cdot\sigma\in\widetilde{\rho}^{(d;m)}_{n}(\pi;\mathcal{B})$. Since $\mathrm{Ad}(s^{\oplus m})[a^{\widetilde{\oplus}n}]=a^{\widetilde{\oplus}n}$, we see that
    \begin{align*}
        r^{(d;m)}(a^{\widetilde{\oplus}n};\mathrm{Ad}(s^{\oplus m})[B])^{-1}
        &=\left[\begin{smallmatrix}
            \ast&\ast&\cdots&\ast\\
            \ast&s^{\oplus d}\zeta_{1}s^{-1}&\cdots&s^{\oplus d}\zeta_{m-d}s^{-1}
        \end{smallmatrix}\right].
    \end{align*}
    Therefore,
    \begin{align*}
        \widetilde{\mathfrak{R}}^{(d;m)}(\pi;\mathcal{B}|v,u)(s\cdot\sigma)
        &=s\left[B_{v,m-d+1},\dots,B_{v,m}\right](s^{-1})^{\oplus d}\times s^{\oplus d}\zeta_{u}s^{-1}
        =\mathrm{Ad}(s)\left[\widetilde{\mathfrak{R}}^{(d;m)}(\pi;\mathcal{B}|v,u)(\sigma)\right]
    \end{align*}
    as desired.
\end{proof}

\section{More about the Grassmannian resolvent}\label{section_moreabout}
\subsection{A generalization of Voiculescu's resolvent equation}
We will establish a generalization of Voiculescu's resolvent equation \cite[Proposition 1.4]{v00},\cite[Lemma 6.1, Proposition 6.1]{v10} within the present framework of $Gr^{(d;m)}(\mathcal{A})$.

\begin{Thm}\label{thm_nescond}
    Let $(s,t),(u,v)\in[m-d]\times[d]$, and let $\pi\in Gr^{(m-d;m)}_{1}(\mathcal{A})$. Assume that $\widetilde{\rho}^{(d;m)}(\pi;\mathcal{B})$ is $(s,t)$-admissible. Then, we have 
    \begin{equation}
        \left(\widetilde{\Delta}^{(d;m)}_{s,t}\widetilde{\mathfrak{R}}^{(d;m)}(\pi;\mathcal{B}|v,u)\right)(\sigma;\sigma')(X)=-\widetilde{\mathfrak{R}}^{(d;m)}(\pi;\mathcal{B}|v,s)(\sigma)\cdot X\cdot \widetilde{\mathfrak{R}}^{(d;m)}(\pi;\mathcal{B}|t,u)(\sigma')\label{eq_14}
    \end{equation}
    for any $\sigma\in\widetilde{\rho}_{n}(\pi;\mathcal{B})$, $\sigma'\in\widetilde{\rho}_{n'}(\pi;\mathcal{B})$ and $X\in M_{n,n'}(\mathcal{B})$.
\end{Thm}
\begin{proof}
    Fix $\sigma\in\widetilde{\rho}^{(d;m)}_{n}(\pi;\mathcal{B})$, $\sigma'\in\widetilde{\rho}^{(d;m)}_{n'}(\pi;\mathcal{B})$, and choose arbitrarily representatives $a=[a_{i,j}]_{i,j=1}^{m}\in\pi$, $B=[B_{i,j}]_{i,j=1}^{m}\in\sigma$, $B'=[B'_{i,j}]_{i,j=1}^{m}\in\sigma'$, where $a_{i,j}\in\mathcal{A}$, $B_{i,j}\in M_{n}(\mathcal{B})$ and $B_{i,j}'\in M_{n'}(\mathcal{B})$. By the $(s,t)$-admissibility and property (3) in Proposition \ref{lem_resolvent_fuuly_matricial}, we have $(\sigma;\sigma')_{s,t}(X)=(B;B')_{s,t}(X)/\sim^{(d;m)}_{n+n'}$ belongs to $\widetilde{\rho}_{n+n'}(\pi;\mathcal{B})$ for any $X\in M_{n,n'}(\mathcal{B})$ (see equation (\ref{eq_6})). Set
    \begin{equation}
    r^{(d;m)}(a^{\widetilde{\oplus}n};B)^{-1}
    =\left[
    \begin{smallmatrix}
        \ast&\ast&\cdots&\ast\\
        \ast&\zeta_{1}&\cdots&\zeta_{m-d}
    \end{smallmatrix}
    \right],\quad
    r^{(d;m)}(a^{\widetilde{\oplus}n};B')^{-1}
    =\left[
    \begin{smallmatrix}
        \ast&\ast&\cdots&\ast\\
        \ast&\zeta_{1}'&\cdots&\zeta_{m-d}'
    \end{smallmatrix}
    \right],\label{equation_resolvent}
    \end{equation}
    where $\zeta_{1},\dots,\zeta_{m-d}\in M_{d,1}(M_{n}(\mathcal{A}))$ and $\zeta_{1}',\dots,\zeta_{m-d}'\in M_{d,1}(M_{n'}(\mathcal{A}))$. Then, there exists an element $U\in GL_{m(n+n')}(\mathcal{R})$ such that
    \begin{align*}
        \,&
        r^{(d;m)}
        \left(
        a^{\widetilde{\oplus}(n+n')};(B;B')_{s,t}(X)
        \right)=
        U^{-1}
        \left[
        \begin{array}{c|c}
            r^{(d;m)}(a^{\widetilde{\oplus}n};B)&
            \begin{smallmatrix}
                0_{(d+s-1)n,(d+s-1)n'} & 0_{(d+s-1)n,dn'}\\
                0_{n,(m-d)n'}          & X[B'_{t,m-d+1},\dots,B_{t,m}']\\
                0_{(m-d-s)n,(m-d)n'}   & 0_{(m-d-s)n,dn'}
            \end{smallmatrix}
            \\ \hline
            0&r^{(d;m)}(a^{\widetilde{\oplus}n};B')
        \end{array}
        \right]U.
    \end{align*}
    Hence, combining this with equation (\ref{equation_resolvent}), we observe that
    \begin{align*}
        r^{(d;m)}
        \left(
        a^{\widetilde{\oplus}(n+n')};(B;B')_{s,t}(X)
        \right)^{-1}
        =
        \left[
        \begin{matrix}
            \ast&\ast&\cdots&\ast\\
            \ast&\eta_{1}&\cdots&\eta_{m-d}
        \end{matrix}
        \right],
    \end{align*}
    where 
    \[
    \eta_{j}
    =
    \left[\begin{smallmatrix}
        (\zeta_{j})_{1}&(-\zeta_{s}X\widetilde{\mathfrak{R}}^{(d;m)}(\pi;\mathcal{B}|t,j)(\sigma'))_{1}\\
        0&(\zeta'_{j})_{1}\\
        (\zeta_{j})_{2}&(-\zeta_{s}X\widetilde{\mathfrak{R}}^{(d;m)}(\pi;\mathcal{B}|t,j)(\sigma'))_{2}\\
        0&(\zeta'_{j})_{2}\\
        \vdots&\vdots\\
        (\zeta_{j})_{d}&(-\zeta_{s}X\widetilde{\mathfrak{R}}^{(d;m)}(\pi;\mathcal{B}|t,j)(\sigma'))_{d}\\
        0&(\zeta'_{j})_{d}\\
    \end{smallmatrix}\right],\quad
    v=
    \begin{bmatrix}
        (v)_{1}\\
        (v)_{2}\\
        \vdots\\
        (v)_{d}
    \end{bmatrix}\in M_{d,1}(M_{*}(\mathcal{A}))
    \]
    for each $v\in\left\{\zeta_{j},\zeta'_{j},-\zeta_{s}X\widetilde{\mathfrak{R}}^{(k,m)}(\pi;\mathcal{B}|t,j)(\sigma')\right\}$ with appropriate size $*$. Therefore, we obtain
    \begin{align*}
        \widetilde{\mathfrak{R}}^{(d;m)}(\pi;\mathcal{B}|v,u)\left((\sigma;\sigma')_{s,t}(X)\right)
        &=
        \left[
        \begin{smallmatrix}
            B_{v,m-d+1}&0&\cdots&B_{v,m}&0\\
            0&B_{v,m-d+1}'&\cdots&0&B_{v,m}'
        \end{smallmatrix}
        \right]\times\eta_{u}\\
        &=\left[
        \begin{smallmatrix}
            \widetilde{\mathfrak{R}}^{(d;m)}(\pi;\mathcal{B}|v,u)(\sigma)
            &-\widetilde{\mathfrak{R}}^{(d;m)}(\pi;\mathcal{B}|v,s)(\sigma)X\widetilde{\mathfrak{R}}^{(d;m)}(\pi;\mathcal{B}|t,u)(\sigma')\\
            0&\widetilde{\mathfrak{R}}^{(d;m)}(\pi;\mathcal{B}|v,u)(\sigma')
        \end{smallmatrix}
        \right].
    \end{align*}
    This implies the desired equality (\ref{eq_14}) by the definition of $\widetilde{\Delta}^{(d;m)}_{s,t}$.
\end{proof}

\subsection{Aspects of the spectral analysis of unbounded operators}
Let $\mathcal{H}$ be a (separable) complex Hilbert space, and let $1$ denote the identity operator on $\mathcal{H}$. 
A bounded operator $a\in B(\mathcal{H})$ is regarded as an affine element 
$\left[
\begin{smallmatrix}
    0&1\\
    1&a
\end{smallmatrix}
\right]/\sim^{(1;2)}_{1}$; that is, we identify $a$ with its graph $\{(1x,ax)=(x,ax)\,|\,x\in\mathcal{H}\}$.

Let $B_{0}(\mathcal{H})$ denote the set of all pure contractions on $\mathcal{H}$ (i.e., operators $a\in B(\mathcal{H})$ satisfying that $\|ax\|<\|x\|$ for all $x\in\mathcal{H}$), and let $C_{0}(\mathcal{H})$ denote the set of all densely defined closed operators on $\mathcal{H}$. The following is well known: \textit{There is a bijection $\Gamma$ from $B_{0}(\mathcal{H})$ onto $C_{0}(\mathcal{H})$ such that the domain of $\Gamma(a)$ is $\mathrm{ran}((1-a^*a)^{\frac{1}{2}})$ and $\Gamma(a)(1-a^*a)^{\frac{1}{2}}=a$. Equivalently, the graph of $t$ is given by $\{((1-a^*a)^{\frac{1}{2}}x,ax)\,|\,x\in\mathcal{H}\}$} (see e.g., \cite{k78}). Then, using \textit{Halmos' unitary dilation}, every element $t\in C_{0}(\mathcal{H})$ can be regarded as an element of $Gr_{1}(B(\mathcal{H}))=Gr^{(1;2)}_{1}(B(\mathcal{H}))$ via 
\[
\pi(t)=\left[
\begin{smallmatrix}
    -a^*&(1-a^*a)^{\frac{1}{2}}\\
    (1-aa^*)^{\frac{1}{2}}&a
\end{smallmatrix}
\right]/\sim^{(1;2)}_{1},\quad
\mbox{where $a\in B_{0}(\mathcal{H})$ satisfies $\Gamma(a)=t$}.
\]  

Let $\mathcal{B}$ be a unital subalgebra of $B(\mathcal{H})$.
For a densely defined closed operator $t\in C_{0}(\mathcal{H})$ corresponding to a pure contraction $a\in B_{0}(\mathcal{H})$, the Grassmannian $\mathcal{B}$-resolvent set of $t$ can be described in terms of the resolvent set $\rho(a;C^*(a,\mathcal{B}))$ of $a$ inside $C^*(a,\mathcal{B})$, the unital $C^*$-algebra generated by $a$ and $\mathcal{B}$. More precisely, \textit{for any $\beta\in\mathcal{B}$,
\begin{equation}\widetilde{\beta}\in\widetilde{\rho}(\pi(t);\mathcal{B})\quad\mbox{if and only if}\quad\beta\cdot(1-a^*a)^{\frac{1}{2}}\in\rho(a;C^*(a,\mathcal{B})).\label{equivalence_colsedresol_grassresol}
\end{equation}
In this case, $\beta-t$ admits a bounded inverse, which is explicitly given by 
\begin{equation}
    (\beta-t)^{-1}=\widetilde{\mathfrak{R}}(\pi(t);\mathcal{B})(\widetilde{\beta})=(1-a^*a)^{\frac{1}{2}}(\beta\cdot(1-a^*a)^{\frac{1}{2}}-a)^{-1},
\end{equation}}
This equivalence follows from a direct computation involving block matrices of the form:
\begin{equation}
\left[\begin{smallmatrix}
    b&1\\
    a&\beta
\end{smallmatrix}\right]
\left[\begin{smallmatrix}
    1&0\\
    -b&1
\end{smallmatrix}\right]=
\left[\begin{smallmatrix}
    0&1\\
    a-\beta b&\beta
\end{smallmatrix}\right],\quad\mbox{where }a,b\in\mathcal{A}.\label{equation_unboundedresolvent}
\end{equation}
This establishes a precise correspondence between the classical resolvent of an unbounded operator and the Grassmannian resolvent. 

Being inspired by this observation, we will show a partial converse of Theorem \ref{thm_nescond} in the case of $(d;m)=(1;2)$ as follows.
\begin{Thm}\label{thm_partial_sufcond}
    Let $\mathcal{R}$ be a unital commutative ring and $\mathcal{B}\subset\mathcal{A}$ an inclusion of unital $\mathcal{R}$-algebras. Let $\pi=
    \left[\begin{smallmatrix}
        \ast&b\\
        \ast&a
    \end{smallmatrix}\right]/\sim^{(1;2)}_{1}\in Gr_{1}(\mathcal{A})$, $b'\in\mathcal{A}$ an element such that $b'b=1$ and $f$ an nc function from an admissible nc set $\Omega$ over $Gr(\mathcal{B})$ to $M(\mathcal{A})$. Assume that for any $n\in\mathbb{Z}_{\geq1}$ and any $\beta\in\Omega\cap M_{n}(\mathcal{B})$ there exists an element $c(\beta)\in GL_{n}(\mathcal{A})$ such that $f(\widetilde{\beta})=(b\otimes I_{n})c(\beta)$. If $f$ satisfies the following difference-differential equation
    \begin{equation}
        \begin{split}
        \begin{aligned}
        \begin{cases}
            (\widetilde{\Delta}f)(\sigma;\sigma')(X)=-f(\sigma)Xf(\sigma')\mbox{ for any }\sigma,\sigma'\in\Omega,\\
            f(\widetilde{\beta_{0}})=
            b\cdot(\beta_{0}b-a)^{-1}
            \mbox{ for some }\beta_{0}\in\Omega\cap\mathcal{B},
        \end{cases}
        \end{aligned}
        \end{split}\label{differential_equation}
    \end{equation}
    then we have $\Omega\cap M(\mathcal{B})\subset\widetilde{\rho}(\pi;\mathcal{B})\cap M(\mathcal{B})$ and $f=\widetilde{\mathfrak{R}}(\pi;\mathcal{B})$ on $\Omega\cap M(\mathcal{B})$.
\end{Thm}
\begin{proof}
    Since $b'b=1$ and $f(\widetilde{\beta})=(b\otimes I_{n})c(\beta)$, we have 
    \begin{equation}
        c(\beta)=(b'\otimes I_{n})f(\widetilde{\beta})\label{equation_cbprimef}
    \end{equation}
    for any $n\in\mathbb{Z}_{\geq1}$ and $\beta\in\Omega\cap M_{n}(\mathcal{B})$. 
    Hence, $c$ is an nc function from $\Omega\cap M(\mathcal{B})$ to $M(\mathcal{A})$ in the sense of \cite{kvv14}. 
    Using the difference-differential equation (\ref{differential_equation}), we have
    \begin{equation}
        c\left(\left[\begin{smallmatrix}
            \beta_{1}&X\\
            0&\beta_{2}
        \end{smallmatrix}\right]\right))
        =\left[
        \begin{smallmatrix}
            c(\beta_{1})&-c(\beta_{1})X(b\otimes I_{n_{2}})c(\beta_{2})\\
            0&c(\beta_{2})
        \end{smallmatrix}
        \right]\label{equation_dd_of_c}
    \end{equation}
    for any $\beta_{1}\in \Omega_{n_{1}}$, $\beta_{2}\in \Omega_{n_{2}}$ and $X\in M_{n_{1},n_{2}}(\mathcal{B})$. Using equation (\ref{equation_dd_of_c}) with the invertibility of $c$, we can see that $(\Delta c^{-1})(\beta_{1};\beta_{2})(X)=X(b\otimes I_{n_{2}})$ for any $\beta_{1}\in \Omega\cap M_{n_{1}}(\mathcal{B})$, $\beta_{2}\in \Omega\cap M_{n_{2}}(\mathcal{B})$ and $X\in M_{n_{1},n_{2}}(\mathcal{B})$. Hence, using the (affine) first order difference-differential formulae \cite[Theorem 2.10]{kvv14}, we have 
    \begin{equation}
        c^{-1}(\beta)-\beta(b\otimes I_{n})=c^{-1}(\beta')-\beta'(b\otimes I_{n})\label{equation_cinverse_difference}
    \end{equation} 
    for any $n\in\mathbb{Z}_{\geq1}$ and $\beta,\beta'\in \Omega\cap M_{n}(\mathcal{B})$.
    In particular, we have 
    \begin{equation}
        c^{-1}(\beta)-\beta b=c^{-1}(\beta_{0})-\beta_{0}b\label{equation_cinverse_difference1}
    \end{equation}
    for any $\beta\in \Omega\cap \mathcal{B}$, where $\beta_{0}$ is given in the difference-differential equation (\ref{differential_equation}). 
    Since $f(\widetilde{\beta_{0}})=b\cdot(\beta_{0}b-a)^{-1}$ and equation (\ref{equation_cbprimef}), we have $c(\beta_{0})=(\beta_{0}b-a)^{-1}$. Using this with equation (\ref{equation_cinverse_difference1}), we have obtained that $c(\beta)=(\beta b-a)^{-1}$ for any $\beta\in\Omega\cap\mathcal{B}$. Hence, combining the direct-sum-preserving property of $c$ with equation (\ref{equation_cinverse_difference}), we have $c(\beta)=(\beta(b\otimes I_{n})-a\otimes I_{n})^{-1}$ for any $\beta\in\Omega\cap M_{n}(\mathcal{B})$. 
    Therefore, we obtain that $f(\widetilde{\beta})=(b\otimes I_{n})\cdot (\beta\cdot(b\otimes I_{n})-a\otimes I_{n})^{-1}$ for any $n\in\mathbb{Z}_{\geq1}$ and $\beta\in \Omega\cap M_{n}(\mathcal{B})$. Moreover, we can see that for any $n\in\mathbb{Z}_{\geq1}$ and $\beta\in\Omega\cap M_{n}(\mathcal{B})$, $\widetilde{\beta}\in\widetilde{\rho}(\pi;\mathcal{B})\cap M_{n}(\mathcal{B})$ if and only if $\beta\cdot(b\otimes I_{n})-a\otimes I_{n}\in GL_{n}(\mathcal{A})$, based on the same computation (\ref{equation_unboundedresolvent}). This implies $\Omega\cap M(\mathcal{B})\subset\widetilde{\rho}(\pi;\mathcal{B})\cap M(\mathcal{B})$. We can then check that $\widetilde{\mathfrak{R}}(\pi;\mathcal{B})(\widetilde{\beta})=(b\otimes I_{n})\cdot (\beta\cdot(b\otimes I_{n})-a\otimes I_{n})^{-1}$ for any $n\in\mathbb{Z}_{\geq}$ and $\beta\in\Omega\cap M_{n}(\mathcal{B})$, as desired.
\end{proof}

\begin{Rem}
    In Theorem \ref{thm_partial_sufcond}, the element $b'$ is assumed to belong to $\mathcal{A}$. However, in the unbounded operator example discussed above, the natural candidate for $b'$ is the ``unbounded'' inverse $(1-a^*a)^{-\frac{1}{2}}$ of $(1-a^*a)^{\frac{1}{2}}$, which does not belong to $B(\mathcal{H})$ in general.  
    For this reason, Theorem \ref{thm_partial_sufcond} cannot be directly applied to the case when $B(\mathcal{H})=\mathcal{A}$ and $t$ is unbounded. Nevertheless, Theorem \ref{thm_partial_sufcond} still provides a conceptual explanation even in this setting. The key point is that \textit{what is essential is not the condition $b'\in\mathcal{A}$, but rather that the products $b'b$ and $b'f$ are well defined.} Indeed, let us focus on the case of $\mathcal{A}=B(\mathcal{H})$. If we assume that $b$ is injective and $b'$ is its (possibly unbounded) inverse (so that $b'$ no longer belongs to $\mathcal{A}$), then an argument entirely analogous to the proof of Theorem \ref{thm_partial_sufcond} still applies. In other words, the algebraic mechanism underlying the theorem remains valid as long as the relevant operator products are well defined, even if $b'$ itself lies outside the ambient algebra.
\end{Rem}

\section{Questions}\label{section_question}
In this section, we will raise several questions. 
Let $\mathcal{A}$ be a unital algebra over $\mathbb{C}$, and $\partial$ be a derivation $\mathcal{A}\to\mathcal{A}^{\otimes2}$ such that $\partial X=1\otimes1$ for some $X\in\mathcal{A}$. Voiculescu \cite{v00} proved that for any invertible element $a\in\mathcal{A}$, $\partial a=a\otimes a$ if and only if there exists an element $b\in\ker(\partial)$ such that $a=(b-X)^{-1}$. Theorem \ref{thm_nescond} can be regarded as the counterpart of the ``only if'' direction of this result. It is then natural to ask the counterpart of the ``if'' direction.
\begin{Que}\label{question_difference_differential_equation}
    Let $\mathcal{R}$ be a unital commutative ring and $\mathcal{B}\subset\mathcal{A}$ an inclusion of unital $\mathcal{R}$-algebras. Let $\Omega$ be a $(u,v)$-right admissible nc set over $Gr^{(d;m)}(\mathcal{B})$ for any $(u,v)\in[m-d]\times[d]$ and $\mathbf{f}=(f_{v,u})_{(v,u)\in[d]\times[m-d]}$ a family of nc functions from $\Omega$ to $M(\mathcal{A})$. Assume that $\mathbf{f}$ satisfies the following system of difference-differential equations:
    \begin{equation}
            (\widetilde{\Delta}^{(d;m)}_{s,t}f_{v,u})(\sigma';\sigma'')(X)=-f_{v,s}(\sigma')\cdot X\cdot f_{t,u}(\sigma'')
    \end{equation}
    for any $(s,t),(u,v)\in[m-d]\times[d],n',n''\in\mathbb{Z}_{\geq1}\mbox{, }\sigma'\in\Omega_{n'},\sigma''\in\Omega_{n''}$ and $X\in M_{n',n''}(\mathcal{B})$.
    Then, what are additional conditions on $\Omega$ or $\mathbf{f}$ so that there is an element $\pi\in Gr^{(m-d;d)}_{1}(\mathcal{A})$ such that $\Omega=\widetilde{\rho}^{(d;m)}(\pi;\mathcal{B})$ and $f_{v,u}=\widetilde{\mathfrak{R}}^{(d;m)}(\pi;\mathcal{B}|v,u)$ for any $(u,v)\in[m-d]\times [d]$ ? 
\end{Que}
Theorem \ref{thm_partial_sufcond} can be viewed as an attempt to answer to Question \ref{question_difference_differential_equation}. However, at present, we do not have a method to treat elements of $Gr^{(d;m)}(\mathcal{A})$ that do not arise from unbounded operators. 

\medskip
Agler-McCarthy-Young \cite{amy18} introduced the notion of \textit{nc manifold}. Then, we are naturally interested in the relation between their nc manifold and the nc Grassmannian. Hence, we have the following question:  
\begin{Que}
    Can we equip the nc $(d;m)$-Grassmanniann $Gr^{(d;m)}(\mathbb{C})$ with an nc manifold structure in the sense of \cite{amy18} ?
\end{Que}

In section \ref{section_moreabout}, we observed some relation between the Grassmannian resolvent of a densely-defined closed operator $t$ and its usual resolvent. We have seen that $\widetilde{\rho}(\pi(t);\mathcal{B})\cap\mathcal{B}\subset\rho(t;\mathcal{B})$. Then, the following question would be natural:

\begin{Que}
    Does it hold that $\widetilde{\rho}(\pi(t);\mathcal{B})\cap \mathcal{B}=\rho_{1}(t;\mathcal{B})$ ?
    Also, what kind of information of $t$ is contained in the higher dimensional resolvent set, $\bigsqcup_{n\geq2}\left(\widetilde{\rho}(\pi(t);\mathcal{B})\cap M_{n}(\mathcal{B})\right)$ ? 
\end{Que}

In the case when $\mathcal{B}=\mathbb{C}$, $\mathcal{H}=L^2(\mathbb{R})$ and $t$ is the multiplication operator $M_{p}$ by a polynomial function $p$ on $\mathbb{R}$, which is unitarily equivalent to the closure of the differential operator $p(-i\partial_{x})$, we can show that $\widetilde{\rho}(t;\mathbb{C})\cap \mathbb{C}=\rho(t)(=\rho_{1}(t;\mathbb{C}))$.

\begin{appendix}
\section{The nc flag manifold $F\ell^{(\mathbf{d};m)}(\mathcal{A})$ over $\mathcal{R}$-algebra $\mathcal{A}$}\label{appendix_flag}
In this section, we will explain how to extend the Grassmannian framework of nc functions, which we have developed so far, to those in ``flag manifold''.
\subsection{}

Let $\mathcal{R}$ be a unital commutative ring and $\mathcal{A}$ a unital $\mathcal{R}$-algebra.
Take arbitrarily $m\in\mathbb{Z}_{\geq1}$ and $\mathbf{d}=(d_{1},d_{2},\dots,d_{k})\in(\mathbb{Z}_{\geq1})^k$ with $d_{1}<d_{2}<\cdots<d_{k}<m$. Set
\begin{equation}
    \begin{split}
        \begin{aligned}
            LT^{(\mathbf{d})}_{n}(\mathcal{A})&=\left\{\left[
    \begin{smallmatrix}
        Z_{k} &       &      0     \\
        \ast  & \ddots &           \\
        \ast  & \ast   &   Z_{1} 
    \end{smallmatrix}
    \right]\,\middle|\,Z_{j}\in GL_{d_{j}-d_{j-1}}(M_{n}(\mathcal{A}))\mbox{ for any $j\in[k]$ $(d_{0}:=0)$}
    \right\},\\
            H^{(\mathbf{d},m)}_{n}(\mathcal{A})&=\left\{
    \left[
    \begin{smallmatrix}
        X&0\\
        Y&Z
    \end{smallmatrix}
    \right]\,\middle|\,X\in GL_{m-d_{k}}(M_{n}(\mathcal{A})),Y\in M_{d_{k},m-d_{k}}(M_{n}(\mathcal{A})),Z\in LT^{(\mathbf{d})}_{n}(\mathcal{A})
    \right\}.
        \end{aligned}
    \end{split}
\end{equation}

The equivalent relation $\sim^{(\mathbf{d};m)}_{n}$ on $GL_{m}(M_{n}(\mathcal{A}))$ is defined by, for any $A,B\in GL_{m}(M_{n}(\mathcal{A}))$, $A\sim^{(\mathbf{d};m)}_{n}B$ if and only if there exists an element $\Gamma\in H^{(\mathbf{d},;m)}_{n}(\mathcal{A})$ such that $A=B\Gamma$.

\medskip
    We define the \textit{nc flag manifold} $F\ell^{(\mathbf{d};m)}(\mathcal{A})$ over $\mathcal{A}$ by
    \[
    F\ell^{(\mathbf{d};m)}(\mathcal{A}):=\bigsqcup_{n\geq1}F\ell^{(\mathbf{d};m)}_{n}(\mathcal{A})\mbox{ with }F\ell^{(\mathbf{d};m)}_{n}(\mathcal{A}):=GL_{m}(M_{n}(\mathcal{A}))/\sim^{(\mathbf{d};m)}_{n}.
    \]

Then, we can embed the affine space $M_{n}(\mathcal{A})^{[\mathbf{d}]}$ (with $[\mathbf{d}]:=\sum_{j=1}^{k}(m-d_{j})(d_{j}-d_{j-1})$, $d_{0}=0$) into $F\ell^{(d;m)}(\mathcal{A})$ in a way that
\begin{equation}
    \begin{split}
        \begin{aligned}
            M_{n}(\mathcal{A})^{[\mathbf{d}]}
            &\ni X
            =\left[
            \begin{smallmatrix}
                0&0&\cdots&0&X(1,1)\\
                0&0&\cdots&X(2,2)&X(2,1)\\
                \vdots&\vdots&\cdot^{\cdot^{\cdot}}&\vdots&\vdots\\
                0&X(k-1,k-1)&\cdots&X(k-1,2)&X(k-1,1)\\
                X(k,k)&X(k,k-1)&\cdots&X(k,2)&X(k,1)
            \end{smallmatrix}
            \right]\\
            &\quad\mapsto \widetilde{X}:=
            \left[
            \begin{smallmatrix}
                0&0&0&0&I_{d_{1}}\otimes I_{n}\\
                0&0&0&I_{d_{2}-d_{1}}\otimes I_{n}&X(1,1)\\
                \vdots&\vdots&\cdot^{\cdot^{\cdot}}&\vdots&\vdots\\
                0&I_{d_{k}-d_{k-1}}\otimes I_{n}&\cdots&X(k-1,2)&X(k-1,1)\\
                I_{m-d_{k}}\otimes I_{n}&X(k,k)&\cdots&X(k,2)&X(k,1)
            \end{smallmatrix}
            \right]
            /\sim^{(\mathbf{d};m)}_{n}\in F\ell^{(\mathbf{d};m)}_{n}(\mathcal{A}),
        \end{aligned}
    \end{split}
\end{equation}
where $X(i,j)\in M_{d_{i+1}-d_{i},d_{j}-d_{j-1}}(M_{n}(\mathcal{A}))$ with $d_{0}:=0$ and $d_{k+1}:=m$.

\medskip
Let us take an element $\sigma\in F\ell^{(\mathbf{d};m)}_{n}(\mathcal{A})$. Then, from definition, it is clear that $A\sim^{(d_{j};m)}_{n}B$ for any $A,B\in\sigma$ and $j\in[k]$. Therefore, we have the well-defined map $p_{d_{j}}:F\ell^{(\mathbf{d};m)}_{n}(\mathcal{A})\ni A/\sim^{(\mathbf{d};m)}_{n}\mapsto A/\sim^{(d_{j};m)}_{n}\in Gr^{(d_{j};m)}_{n}(\mathcal{A})$ for any $j\in[k]$.

\subsection{}
The notions of nc set and nc function over $Gr^{(d;m)}(\mathcal{A})$ (see Definitions \ref{def_nc_operation} and \ref{def_nc_subset_function}) are naturally generalized to $F\ell^{(\mathbf{d};m)}(\mathcal{A})$. We can also generalize the \textit{Grassmannian intertwining property} (see Proposition \ref{prop_grassmann_intertwining}) to the present setting by the exactly same computation.
\begin{Prop}
    Let $\Omega$ be an nc set over $F\ell^{(\mathbf{d};m)}(\mathcal{A})$ and $f$ be a graded function from $\Omega$ to $M(\mathcal{B})$. Then, the following are equivalent:
    \begin{enumerate}
        \item[(a)] $f$ is nc.
        \item[(b)] $f$ satisfies the following condition:
        \begin{enumerate}
            \item[(I)] For any $n,n'\in\mathbb{Z}_{\geq1}$, $T\in M_{n,n'}(\mathcal{R})$ and $\sigma\in \Omega_{n}$, $\sigma'\in\Omega_{n'}$, we have
        \[
        \left[\begin{smallmatrix}
            I_{n}&T\\
            0&I_{n'}
        \end{smallmatrix}\right]\cdot(\sigma\oplus\sigma')=\sigma\oplus\sigma'\quad
        \Rightarrow\quad
        f(\sigma)T=Tf(\sigma').
        \]
        \end{enumerate}
    \end{enumerate}
\end{Prop}

\subsection{}
The \textit{nc $(j;u,v)$-difference-differential operator} $\widetilde{\Delta}^{(\mathbf{d};m)}_{j;u,v}$ with $j\in[k]$ and $(u,v)\in[m-d_{j}]\times[d_{j}-d_{j-1}]$ is defined by $\widetilde{\Delta}^{(\mathbf{d};m)}_{j;u,v}:=\widetilde{\Delta}^{(d_{j};m)}_{u,d_{j-1}+v}$.

\subsection{}
In section \ref{section_Grassmann_resolvent}, we introduced the $(d;m)$-Grassmannian $\mathcal{B}$-resolvents. One possibility of its flag manifold extension, which has Grassmannian resolvent as parts, is the following:

Let $\mathbf{d}\in(\mathbb{Z}_{\geq1})^{k}$ with $d_{1}<d_{2}<\cdots<d_{k}<m$ with $k\in[m-1]$ and $\pi\in F\ell^{(m-\mathbf{d};m)}_{1}(\mathcal{A})$ with $m-\mathbf{d}:=(m-d_{k},m-d_{k-1},\dots,m-d_{1})$. Then, we define the \textit{$(\mathbf{d};m)$-flag $\mathcal{B}$-resolvent set} $\widetilde{\rho}^{(\mathbf{d};m)}(\pi;\mathcal{B})=\bigsqcup_{n\geq1}\widetilde{\rho}^{(\mathbf{d};m)}_{n}(\pi;\mathcal{B})$ in a way that
\begin{equation}
    \widetilde{\rho}^{(\mathbf{d};m)}_{n}(\pi;\mathcal{B})
    :=\left\{\sigma\in F\ell^{(\mathbf{d};m)}_{n}(\mathcal{B})\,\middle|\,p_{d_{j}}(\sigma)\in\widetilde{\rho}^{(d_{j};m)}_{n}(p_{m-d_{j}}(\pi);\mathcal{B})\mbox{ for any }j\in[k]\right\}.
\end{equation}

We define the \textit{$(\mathbf{d};m)$-flag $\mathcal{B}$-resolvent $\widetilde{\mathfrak{R}}^{(\mathbf{d};m)}(\pi;\mathcal{B}|j;v,u)$ of $\pi$ with respect to $(j;v,u)$} by
\begin{equation}
    \widetilde{\mathfrak{R}}^{(\mathbf{d};m)}(\pi;\mathcal{B}|j;v,u)(\sigma):=
    \widetilde{\mathfrak{R}}^{(d_{j};m)}(p_{m-d_{j}}(\pi);\mathcal{B}|d_{j-1}+v,u)(p_{d_{j}}(\sigma)),\quad\sigma\in\widetilde{\rho}^{(\mathbf{d};m)}(\pi;\mathcal{B})
\end{equation}
for any $j\in[k]$ and $(u,v)\in[m-d_{j}]\times[d_{j}-d_{j-1}]$.

\subsection{}
By Theorem \ref{thm_nescond}, we have the following:
\begin{Cor}
    For any $j\in[k]$ and $(s,t),(u,v)\in[m-d_{j}]\times[d_{j}-d_{j-1}]$, 
    We have
    \begin{equation}
        \left(\widetilde{\Delta}^{(\mathbf{d};m)}_{j;s,t}\widetilde{\mathfrak{R}}^{(\mathbf{d};m)}(\pi;\mathcal{B}|j;v,u)\right)(\sigma;\sigma')(X)
        =-\widetilde{\mathfrak{R}}^{(\mathbf{d};m)}(\pi;\mathcal{B}|j;v,s)(\sigma)\cdot X\cdot\widetilde{\mathfrak{R}}^{(\mathbf{d};m)}(\pi;\mathcal{B}|j;t,u)(\sigma')
    \end{equation}
    for any $\sigma\in\widetilde{\rho}^{(\mathbf{d};m)}_{n}(\pi;\mathcal{B})$, $\sigma' \in\widetilde{\rho}^{(\mathbf{d};m)}_{n'}(\pi;\mathcal{B})$ and $X\in M_{n,n'}(\mathcal{B})$.
\end{Cor}
\end{appendix}

\section*{Acknowledgements}
The author gratefully acknowledges his supervisor, Professor Yoshimichi Ueda, for his encouragement. The author also thanks Dr. Kenta Kojin for enjoyable discussions.
This work was partially supported by JST SPRING (Grant Number JPMJSP2125) and by JSPS Research Fellowship for Young Scientists (KAKENHI Grant
Number JP25KJ1405).
}

\end{document}